\documentclass[10pt]{amsart}
\usepackage{amssymb}
\usepackage{amsmath}
\usepackage{booktabs}

\numberwithin{equation}{section}
\theoremstyle{plain}
\newtheorem{theorem}{Theorem}[section]
\newtheorem{proposition}[theorem]{Proposition}
\newtheorem{lemma}[theorem]{Lemma}

\theoremstyle{remark}

\newtheorem*{remark}{Remark}

\theoremstyle{definition}
\newtheorem{definition}[theorem]{Definition}

\newcommand{\SU}{\mathrm{SU(2)}}
\newcommand{\SO}{\mathrm{SO(3)}}
\newcommand{\He}{\mathrm{H_3}}
\newcommand{\HQ}{\mathrm{H_3^*}}
\newcommand{\bR}{\mathbb{R}}
\newcommand{\bC}{\mathbb{C}}
\newcommand{\bZ}{\mathbb{Z}}
\newcommand{\AffR}{\mathrm{Aff(\bR)_0}\times\bR}
\newcommand{\AffS}{\mathrm{Aff(\bR)_0}\times S^1}
\newcommand{\Gtwo}{\mathrm{G_{3.2}}}
\newcommand{\Gthree}{\mathrm{G_{3.3}}}
\newcommand{\Gfour}{\mathrm{G}_\mathrm{3.4}^a}
\newcommand{\Gfive}{\mathrm{G}_\mathrm{3.5}^a}
\newcommand{\SEU}{\widetilde{\mathrm{SE}}\mathrm{(2)}}
\newcommand{\SEn}{\mathrm{SE}_n\mathrm{(2)}}
\newcommand{\SE}{\mathrm{SE(2)}}
\newcommand{\AU}{\mathrm{\tilde{A}}}
\newcommand{\An}{\mathrm{A}_n}
\newcommand{\SL}{\mathrm{SL(2,\bR)}}
\newcommand{\PSL}{\mathrm{PSL(2,\bR)}}
\newcommand{\Ad}{\mathrm{Ad}}
\newcommand{\Ml}{\left(\begin{smallmatrix}}
\newcommand{\Mr}{\end{smallmatrix}\right)}
\newcommand{\indsp}{\hspace{2mm}}
\newcommand{\botrule}{\bottomrule}

\title[The chains of CR structures on the 3-dimensional Lie groups]{The chains of left-invariant CR structures on the 3-dimensional Lie groups}

\author{Daiki Maeda}

\address{Graduate School of Mathematical Sciences, The University of Tokyo, 3-8-1 Komaba, Meguro-ku, Tokyo 153-8914 JAPAN}

\email{maeda.daiki.math@gmail.com}



\begin{document}
\begin{abstract}
Let $G$ be a 3-dimensional connected Lie group  with a left-invariant non\-degenerate CR structure.
We show that all the chains on $G$ are closed if and only if $G$ is CR equivalent to one of the 
left-invariant spherical CR structures on $\SU$,  $\SO$ and  $\HQ=$ $\He/\mathbb{Z}$, a quotient group of the Heisenberg group.
\end{abstract}

\maketitle

\section{Introduction}\label{sec1}

The chains on a nondegenerate CR manifold $M$ are a family of curves that are invariantly constructed from the CR structure of $M$. Chains were introduced by \'{E}. Cartan~\cite{car} for hypersurfaces in $\mathbb{C}^{2}$ and generalized to higher dimensions by Chern--Moser~\cite{che} and Fefferman~\cite{feff}; see \cite{bur} for the relation between these.
 Fefferman's definition goes through a conformal class $[g]$ of pseudo-Riemanian metrics defined on a circle bundle $C(M)$ over $M$, which is now called Fefferman space.  Chains are then defined to be the projections onto $M$ of the null geodesics of the metric $g$ in the class $[g]$.

\indent
Chains have been studied for various CR manifolds. For example, Reinhardt hypersurfaces~\cite{jac}, boundaries of Grauert tubes~\cite{ste}, and left-invariant CR structures on $\SU$~\cite{cas}. Castro and Montgomery~\cite{cas} showed that the moduli space of all left-invariant CR structures on $\SU$ are parametrized by a real parameter $a\ge1$ and wrote down the dynamical system for the chains. For $a = 1$, the CR structure can be realized as the standard one on $S^3\subset\mathbb{C}^{2}$. In this case,  any chain is given by the intersection of $S^3$ with a complex line; hence all the chains are closed.  The situation changes for $a>1$: except for finite values of $a>1$, they showed that closed chains and quasiperiodic ones both exist.
In this paper, we complete their result by showing that the chains of both types exist for all $a>1$.

Moreover, we study the chains of all left-invariant CR structures on 3-dimensional Lie groups. \'{E}. Cartan classified 3-dimensional homogeneous CR manifolds in~\cite{car}. 
Bor and Jacobowitz~\cite{bor}  revisited Catran's paper and gave a detailed classification results in modern language for four 3-dimensional Lie groups.  Here we follow their method and give a complete  classification of left-invariant CR structures on 3-dimensional Lie groups.
Then we study the existence of chains that are not closed.

To state our result, let us recall 3 examples of CR structures on Lie groups for which all the chains are closed.  The first one is the sphere $SU(2)\cong S^{3}$, as we described above.  The second one is its universal cover $\SO$ as CR manifold.  The third one is a quotient of the Heisenberg group $\He$.
As $S^{3}$ is the one-point compactification of $\He$ as a CR manifold,  $\He$ has chains that go to infinity.  However, by taking a quotient $\HQ=\He/\mathbb{Z}$, 
see  Section \ref{Classification-CRstructures}, we obtain a noncompact CR manifold whose chains are all closed.  Note that in each case the CR structure is locally equivalent to the one on the sphere and are called spherical CR structure.
Our main result claims that these are complete list of left invariant CR manifolds whose chains are all closed.

\begin{theorem}\label{mainthm}
Let $G$ be a 3-dimensional connected Lie group with a left-invariant nondegenerate CR structure. Then, all the chains on $G$ are closed if and only if $G$ is equivalent to $\SU$, $\SO$, or $\HQ$
with spherical CR structures.
\end{theorem}

Our analysis of the chains based on \cite{cas}.
For a left-invariant CR structure on a Lie group $G$, the Fefferman space is $G\times S^1$ and the metric $g$ can be chosen to be left-invariant on $G\times S^1$. Thus our study of chains  can be 
reduced to the study Hamiltonian systems for the left-invariant metrics on Lie groups.
Such a Hamiltonian system has a reduced system defined on the dual space of the Lie algebra, which is called the Lie--Poisson equation. 
If  the Lie--Poisson equation has solution that are not closed, then there is a chain that are not closed.
If  the Lie--Poisson equation has a stationary solution, then the OED for the chain has explicit solutions that are not closed. 
Except for the case $\SU$, Theorem \ref{mainthm} can be proved by using this criterion. 

For $\SU$, there are the cases where each solution $C$ of the Lie--Poisson equation is closed;  the chains are given as  curves in a $S^{1}$-bundle over $C$.  The Berry phase of a chain is defined as the difference of the angle in the fiber $S^{1}$  when one travels along the chain for the period of $C$.
The chain is closed if and only if the Berry phase is a rational multiple of $\pi$.
Hence the main part of the proof of the theorem is the evaluation of the Berry phase, which are given in 
Section  \ref{sec5}.

\indent Our paper is organized as follows. In Section \ref{sec2}, we recall Fefferman's definition of chains and Hamiltonian systems for left-invariant metrics. In Section \ref{sec3}, we classify left-invariant nondegenerate CR structures on 3-dimensional connected Lie groups and compute  Fefferman metrics. In Section \ref{sec4}, we prove the main theorem except for the case of $\SU$ by solving the Lie--Poisson equation. In Section \ref{sec5}, we study the case $\SU$ and completes the proof of the theorem.

\section{Preliminaries}\label{sec2}
\subsection{CR manifolds}
Let $M$ be a $C^\infty$-manifold of dimension $3$. A \textit{CR structure} on $M$ is a complex line bundle $T^{1,0}\subset\mathbb{C} TM$ such that $ T^{1,0}\cap\overline{T^{1,0}}=\{0\}$.
We assume that there is  a real 1-form $\theta$ such that $\ker\theta=T^{1,0}\oplus\overline{T^{1,0}}\subset\mathbb{C} TM$ and $\theta\wedge d\theta\ne0$.  In this case, $T^{1,0}M$, is said to be \textit{ nondegenerate} and $\theta$ is called a \textit{contact from}.

The \textit{Reeb vector field} $T$ of $\theta$ is a vector field characterized by
\begin{align}\label{Reeb}
\theta(T) = 1,\quad T\,\lrcorner\, d\theta=0.
\end{align}
For a  local frame $T_{1}$ for $T^{1,0}$, we may define  a local frame for $\bC TM$ by $\{T,T_{1},T_{\bar 1}:=\overline{T_{1}}\}$. 
Such a frame is called  an \textit{admissible frame} and its dual coframe $\{\theta,\theta^{1},\theta^{\bar1}\}$ is called an \textit{admissible coframe}. In this coframe, we have
\begin{equation*}
d\theta=ih_{1\bar1}\theta^1\wedge\theta^{\bar1} 
\end{equation*}
for a real function $h_{1\bar1}$, which can be seen as a hermitian metric on $T^{1,0}$.

The \textit{Tanaka--Webster connection} $\nabla$ (cf.~\cite{tan,web}) of  $(T^{1,0},h_{1\bar1})$ 
is a metric connection satisfying the structure equation
\[
d\theta^1 = \theta^1\wedge\omega_1^{\indsp 1} +
A_{\bar1}{}^{1}\theta^{\bar1}\wedge\theta,
\]
where the connection from is defined by
$\nabla T_{1}=\omega_{1}{}^{1}T_{1}$.
The tensor $A_{\bar1}{}^{1}$ is called the \textit{Tanaka--Webster torsion}.
The \textit{Tanaka--Webster scalar curvature} $S$ is then defined by
\[
d\omega_{1}{}^1
 \equiv -i Sd\theta
\mod{\theta}.
\]

We now recall the construction of the Fefferman metric by following Lee~\cite{lee}.
Let
\begin{equation*}
K^{*}(M):=\{\zeta\in\wedge^{2}\bC T^*M\mid \zeta\ne0, \bar{V} \,
\lrcorner\,\zeta=0, \forall V\in T^{1,0}\},
\end{equation*}
the canonical bundle with zero section removed.  The \textit{Fefferman space} is the circle bundle 
 $C(M)=K^*(M)/\mathbb{R}^{+}$,  the quotient by the  natural $\bR^+$ action $\zeta\mapsto\lambda\zeta$. 
 For a choice of an admissible coframe $\{\theta,\theta^{1},\theta^{\bar1}\}$ gives  a local section $\zeta_0=[\theta\wedge\theta^1]$ of the bundle  $\pi\colon C(M)\to M$, which also gives a fiber coordinate $s\in\bR/2\pi\bZ$  by $[\zeta]=e^{ i s}\zeta_0$.
 Then we may define a 1-form on $C(M)$ by
\begin{align}
\label{sigma}
\sigma := \frac{1}{3}\Big(ds-\mathrm{Im}\,\pi^*\omega_1^{\indsp 1} - \frac{S}{4}\pi^*\theta\Big).
\end{align}
It turns out that $\sigma$ is independent of the choice of an admissible coframe and 
so $\sigma$ is defined globally on $C(M)$. The \textit{Fefferman metric} $g_\theta$ corresponding to $\theta$ is the Lorentz  metric on $C(M)$ defined by
\begin{align}
\label{Feffermanmetric}
g_\theta := 2\pi^*(h_{1\bar{1}}\theta^1\odot\theta^{\bar{1}} )+ 4\pi^*\theta\odot\sigma,
\end{align}
where $\theta\odot\sigma := \frac{1}{2}(\theta\otimes\sigma + \sigma\otimes\theta)$. 
It was shown in \cite{lee} that
\[
g_{\hat\theta}=e^\Upsilon g_\theta\quad
\text{if }\hat\theta=e^\Upsilon\theta\text{ with }\Upsilon\in C^\infty(M).
\]
Thus the conformal class $[g_\theta]$ is determined by the CR structure of $M$.
\begin{definition}
A \textit{chain} is a curve $\gamma$ on $M$ which is the projection onto $M$ of a null geodesic $\tilde\gamma$ of $(C(M),[g_\theta])$ and $\theta(\dot\gamma)\neq0$.
\end{definition}

\subsection{Geodesics on Lie groups with left-invariant metrics}\label{sec2.2}
Let $G$ be a Lie group of dimension $n$ and $g$ a left-invariant pseudo-Riemannian metric on $G$. 
Let  $\{e_i\}$ be a basis  the Lie algebra $\mathfrak{g}$ of $G$ and $\{\omega^i\}$ be its dual basis. Then we may trivialize $T^{*}G$ by 
\begin{equation*}
G\times\bR^n\rightarrow T^*G,\, (p,M_i) \mapsto M_i\omega^i(p).
\end{equation*}
We can see $M_{i}$  as functions on $T^{*}G$; via this trivialization, we can also see $e_{i}$ and $\omega^{i}$ as vectors and one forms on $T^{*}G$.
With these conventions, the tautological 1-form $\Theta$ be $T^*G$ can be written as
\begin{equation*}
\Theta = M_i  \omega^i.
\end{equation*}
Thus the canonical 2-from is 
\begin{equation*}
\omega := -d\Theta = -(dM_i\wedge \omega^i+M_id\omega^i).
\end{equation*}
Hence the isomorphism $\omega^{\sharp}\colon T^{*}(TG)\to T(TG)$ induced by $\omega$ satisfies
\begin{align}\label{sharp}
\omega^\#(\omega^i)=-\frac{\partial}{\partial M_i},\quad
\omega^\#\left(dM_i - M_j(e_i\,\lrcorner\,d\omega^j)\right)=e_i.
\end{align}
We also have
\begin{align}\label{Poisson bracket}
M_j(e_i\,\lrcorner\, d\omega^j)=\{M_i,M_k\}\omega^k,
\end{align}
where $\{\cdot,\cdot\}$ is the Poisson bracket on $(T^{*}G,\omega)$.
As $g$ is left-invariant, $g_{ij}:=g(e_i,e_j)$ is a constant matrix. Using the invinverseerise matrix $\{g^{ij}\}$ of $\{g_{ij}\}$, we define the Hamiltonian of $g$ by
\begin{align}\label{Hamiltonian}
H(M_i\omega^i)=\frac{1}{2}g^{ij}M_iM_j=\frac{1}{2}M_{i}M^{i},
\end{align}
where we have set $M^{i}=g^{ij}M_j$.
It follows that $dH = M^{j}dM_j$.
Thus, by using (\ref{sharp}) and (\ref{Poisson bracket}), we may write the Hamilton vector field as
\[
X_H := \omega^\# dH = -M^{j}\{M_j,M_k\}\frac{\partial}{\partial M_k} + M^{j}e_j.
\]
Therefore, an integral curve  $\gamma(t)=M_i(t)\omega^i(p(t))$ of $X_H$ satisfies 
\begin{equation}\label{L-Peq}
\left\{
\begin{array}{lll}
\dot{M_i}&=& -M^{j}\{M_j,M_i\},\\
\dot{p}&=& M^{j}e_j. \\
\end{array}
\right.
\end{equation}
The first system of the equations is called the \textit{Lie--Poisson equation}.

Let $R$ be the right transformation of $G$. Then the map
\begin{equation*}
J:T^*G\to\mathfrak{g}^*, \quad \alpha_p \mapsto R_p^*\alpha_p,
\end{equation*}
is called a momentum map.  It is shown that any integral curve of $X_H$ lies on a level set of $J$;
see e.g. \cite{abr}. 

\section{The Left-invariant CR structures on the 3-dimensional connected Lie groups}\label{sec3}
\subsection{Classification of CR structures}\label{Classification-CRstructures}
We recall the following well-known fact; see e.g.~\cite{oni}. Any 3-dimensional connected real non abelian Lie group is isomorphic to one of the following pairwise nonisomorphic Lie groups:
\begin{align*}
&\AffR, \AffS, \He, \HQ,
\\
& \Gtwo, \Gthree, \Gfour (a\ge0, a\neq1), \Gfive (a>0), 
\\
& \SEU, \SEn, \SE\cong\mathrm{G}_\mathrm{3.5}^0, 
\\
&\AU, \An (n\ge3), \SL, \PSL, \SU, \SO.
\end{align*}
Here
\[
\begin{array}{ccl}
\vspace{2mm}
\AffR & = & \left\{\begin{pmatrix}
1 & 0 \\
x & e^{y} \\
\end{pmatrix} : x,y\in\bR\right\},\\
\vspace{2mm}
\He &=& \left\{\begin{pmatrix}
1 & x & z\\
0 & 1 & y \\
0 & 0 & 1 \\
\end{pmatrix} : x, y, z\in\bR\right\},\\\vspace{2mm}
\Gtwo &=& \left\{\begin{pmatrix}
1 & 0 & 0\\
y & e^{z} & 0 \\
x & -ze^{z} & e^{z} \\
\end{pmatrix} : x, y, z\in\bR\right\},\\\vspace{2mm}
\Gthree&=&\left\{\begin{pmatrix}
1 & 0 & 0\\
y & e^{z} & 0 \\
x & 0 & e^{z} \\
\end{pmatrix} : x, y, z\in\bR\right\},\\\vspace{2mm}
\Gfour&=&\left\{\begin{pmatrix}
1 & 0 & 0\\
x & e^{a z}\cosh z &  -e^{a z}\sinh z\\
y & -e^{a z}\sinh z & e^{a z}\cosh z \\
\end{pmatrix} : x, y, z\in\bR\right\},\\
\Gfive&=&\left\{\begin{pmatrix}
1 & 0 & 0\\
x & e^{a z}\cos z &  -e^{a z}\sin z\\
y & e^{a z}\sin z & e^{a z}\cos z \\
\end{pmatrix} : x, y, z\in\bR\right\}.
\end{array}
\]
The Lie group $\HQ$ is $\He / \bZ(\He(\bZ))$, where 
\[
\bZ(\He(\bZ))=\left\{
\begin{pmatrix}
1 & 0 & z\\
0 & 1 & 0 \\
0 & 0 & 1 \\
\end{pmatrix} : z\in\bZ\right\}.
\]
 The Lie group $\SEU$ is the universal covering group of $\SE$ and $\SEn$ are the $n$-fold coverings. The groups $\AU$ and $\An$are respectively  the universal and  the $n$-fold coverings  of $\PSL$.

We first recall  the normal form of the left-invariant CR structures on $\SU$,
 $\SL$ and $\He$.
  
\begin{proposition}[\cite{cas, bor}]
For $\SU$,  $\SL$ and $\He$,
the left-invariant nondegenerate CR structures are classified as in Table \ref{table:1}. 
The CR structures in Table  \ref{table:1} are pairwise non CR diffeomorphic and
the ones on $\He$,  $\SU$ corresponding to $a=1$, and on $\SL$ corresponding to $a=1,-3+2\sqrt{2}$ are spherical.
\end{proposition}

\begin{table}[ht]
\begin{center}
\begin{minipage}{330pt}
\caption{Normalized left-invariant CR structures on $\He$, $\SL$, $\SU$}\label{table:1}
\begin{tabular}{llll}
\toprule
Group & $T^{1,0}$ & basis & Lie bracket \\
\midrule
$\He$ & $\bC(e_1 - i e_2)$ & 
\begin{tabular}{l}
$e_1 = \Ml
0 & 1 & 0\\
0 & 0 & 0 \\
0 & 0 & 0 \\
\Mr$,\\\addlinespace[2mm]
$e_2 = \Ml
0 & 0 & 0\\
0 & 0 & 1 \\
0 & 0 & 0 \\
\Mr$, \\
\addlinespace[2mm]
$e_3 = \Ml
0 & 0 & 1 \\
0 & 0 & 0 \\
0 & 0 & 0 \\
\Mr$ 
\end{tabular}&
\begin{tabular}{l}
$[e_1, e_2] = e_3$,\\
\addlinespace[2mm]
$[e_2,e_3] = 0$,\\\addlinespace[2mm]
$[e_3,e_1] = 0$
\end{tabular}\\
\midrule
$\SL$ & \begin{tabular}{l}
$\bC \left(ae_1+e_2-i\frac{1+a}{2} e_3\right),$\\
\addlinespace[2mm] $a \in(-1,0)\cup(0,1]$
\end{tabular} & 
\begin{tabular}{l}
$e_1 = \Ml
0 & 1 \\
0 & 0 \\
\Mr$,\\\addlinespace[2mm]
$e_2 = \Ml
0 & 0 \\
1 & 0 \\
\Mr,$ \\
\addlinespace[2mm]
$e_3 = \Ml
1 & 0 \\
0 & -1 \\
\Mr$
\end{tabular}&
\begin{tabular}{l}
$[e_1, e_2] = e_3$,\\\addlinespace[2mm]
$[e_2,e_3] = 2e_2$,\\\addlinespace[2mm]
$[e_3,e_1] = 2e_1$
\end{tabular}\\
\midrule
$\SU$ & $\bC \left(e_1-\frac{i}{a} e_2\right), a \ge 1$ & 
\begin{tabular}{l}
$e_1 = \frac{1}{2}\Ml
0 & i \\
i & 0 \\
\Mr$,\\\addlinespace[2mm]
$e_2 = \frac{1}{2}\Ml
0 & -1 \\
1 & 0 \\
\Mr$, \\\addlinespace[2mm]
$e_3 = \frac{1}{2}\Ml
i & 0 \\
0 & -i \\
\Mr$
\end{tabular}&
\begin{tabular}{l}
$[e_1, e_2] = e_3$,\\\addlinespace[2mm]
$[e_2,e_3] = e_1$,\\\addlinespace[2mm]
$[e_3,e_1] = e_2$
\end{tabular}\\
\botrule
\end{tabular}
\end{minipage}
\end{center}
\end{table}

For other groups, we have the following

\begin{proposition}
The left-invariant nondegenerate CR structures on each Lie group are classified as in Table \ref{table:2}. The CR structures in Table \ref{table:2} are pairwise non CR diffeomorphic and the ones on $\AffR, \AffS$ and $\mathrm{G}_\mathrm{3.4}^3$ are spherical. Any left-invariant CR structure on $\Gthree$ is degenerate.
\end{proposition}

\begin{proof}
For the cases $\AffR$ and $\AffS$, by $\Ad_g$ with $g=\Ml
1 & 0 \\
x & \lambda e^y \\
\Mr$, we may map $e_1$, $e_2$ to $\lambda e^{y} e_1$, $-x e_1 +e_2$, respectively. Thus, any left-invariant nondegenerate CR structure can be mapped to $\bC(e_2+ie_1+ae_3)$ with $\mathrm{Im}\,a\neq0$. For $\AffR$, by a constant multiple of $\bR$, we may further assume  $\mathrm{Im}\,a=1$.
In addition, there exists a Lie algebra homomorphism which sends $e_2$ to $e_2-\mathrm{Re}\,ae_3$; hence we may assume $a=i$. For $\AffS$, the claim follows from the embedding given in the next section.  

For $\Gtwo$, by $\Ad_g$ with $g=\Ml
1 & 0 & 0\\
y & \lambda e^z & 0\\
x & -\lambda ze^z & e^z
\Mr$, we may map $e_1$, $e_2$, $e_3$ to 
\[
\lambda e^{z} e_1, \quad \lambda e^z(-z e_1 + e_2), \quad (y-x)e_1 - ye_2 + e_3,
\]
 respectively. Thus, any Levi distribution is mapped to $\mathrm{Span}\{e_1, e_3\}$ or $\mathrm{Span}\{e_2, e_3\}$. Moreover, the latter structures can be mapped to $\bC(e_3+ie_2)$. The former structures are degenerate.\\
\indent For $\Gthree$, by $\Ad_g$ with $g=\Ml
1 & 0 & 0\\
y & e^z & 0\\
x & 0 & e^z
\Mr$, we may map $e_1$, $e_2$, $e_3$  to 
\[
\lambda e^{z} e_1,\quad 
 e^ze_2,\quad
  -xe_1 - ye_2 + e_3,
  \]
   respectively. Thus, any Levi distribution is mapped to $\mathrm{Span}\{e_1, e_3\}$ or $\mathrm{Span}\{ae_1+e_2, e_3\}$. These CR structure are degenerate.\\
\indent For $\Gfour$, by $\Ad_g$ with $g=\Ml
1 & 0 & 0\\
x & \lambda e^{a z}\cosh z & -\lambda e^{a z}\sinh z\\
y & -\lambda e^{a z}\sinh z & \lambda e^{a z}\cosh z
\Mr$, we may map $e_1$, $e_2$, $e_3$ to 
\[
\lambda e^{a z}(\cosh z e_1- \sinh z e_2),\  \lambda e^{a z}(-\sinh z e_1 + \cosh z e_2),\ 
 (y-a x)e_1 +(x-a y) e_2 + e_3,
\]
 respectively. Thus, any Levi distribution is mapped to $\mathrm{Span}\{e_1, e_3\}$, $\mathrm{Span}\{e_2, e_3\}$, $\mathrm{Span}\{e_1+e_2,e_3\}$, or $\mathrm{Span}\{e_1-e_2,e_3\}$. Moreover, the former two structures can be mapped to $\bC(e_3+ie_1)$, $\bC(e_3+ie_2)$, respectively. The latter two structures are degenerate. In addition, by the Lie algebra homomorphism which swaps $e_1$ and $e_2$, the structures $\bC(e_3+ie_1)$ and $\bC(e_3+ie_2)$ are equivalent.\\
\indent For $\Gfive$, by $\Ad_g$ with $g=\Ml
1 & 0 & 0\\
x & \lambda e^{a z}\cos z & -\lambda e^{a z}\sin z\\
y & \lambda e^{a z}\sin z & \lambda e^{a z}\cos z
\Mr$, we may map $e_1$, $e_2$, $e_3$ to 
\[
\lambda e^{a z}(\cos z e_1 + \sin z e_2), \  \lambda e^z{a z}(-\sin z e_1 + \cos z e_2), \ (y-a x)e_1 - (x+a y) e_2 + e_3,
\]
 respectively. Thus, any Levi distribution is mapped to $\mathrm{Span}\{e_2, e_3\}$. Moreover, these structures can be mapped to $\bC(e_3+ie_2)$.
\end{proof}

\begin{table}[ht]
\begin{center}
\begin{minipage}{\textwidth}
\caption{Normalized left-invariant CR structures on $\AffR$, $\AffS$, $\Gtwo$, $\Gfour$, $\Gfive$}\label{table:2}
\begin{tabular}{llll}
\toprule
Group & $T^{1,0}$ & basis & Lie bracket \\
\midrule
$\AffR$ & $\bC(e_2+i e_1 + i e_3)$ & \begin{tabular}{l}
$e_1 = \Ml
0 & 0 \\
1 & 0 \\
\Mr$,\\\addlinespace[2mm]
$e_2 = \Ml
0 & 0 \\
0 & 1 \\
\Mr$\end{tabular} &
\begin{tabular}{l}
$[e_1, e_2] = -e_1$, \\\addlinespace[2mm]
$[e_2,e_3] = 0$,\\\addlinespace[2mm]
$[e_3,e_1] = 0$
\end{tabular}\\
\midrule
$\AffS$ &\begin{tabular}{l}
$\bC(e_2+i e_1 + a i e_3)$,\\\addlinespace[2mm] $a>0$ \end{tabular} & \begin{tabular}{l}$e_1 = \Ml
0 & 0 \\
1 & 0 \\
\Mr$,\\\addlinespace[2mm] 
$e_2 = \Ml
0 & 0 \\
0 & 1 \\
\Mr$\end{tabular} &
\begin{tabular}{l}
$[e_1, e_2] = -e_1$,\\\addlinespace[2mm]
$[e_2,e_3] = 0$,\\\addlinespace[2mm]
$[e_3,e_1] = 0$
\end{tabular}\\
\midrule
$\Gtwo$ & $\bC(e_3 + i e_2)$ & 
\begin{tabular}{l}
$e_1 = \Ml
0 & 0 & 0\\
0 & 0 & 0 \\
1 & 0 & 0 \\
\Mr$,\\\addlinespace[2mm]
$e_2 = \Ml
0 & 0 & 0\\
1 & 0 & 0 \\
0 & 0 & 0 \\
\Mr$, \\\addlinespace[2mm]
$e_3 = \Ml
0 & 0 & 0 \\
0 & 1 & 0 \\
0 & -1 & 1 \\
\Mr$ 
\end{tabular} &
\begin{tabular}{l}
$[e_1, e_2] = 0$,\\\addlinespace[2mm]
$[e_2,e_3] = e_1 - e_2$, \\\addlinespace[2mm]
$[e_3,e_1] = e_1$
\end{tabular}\\
\midrule
$\Gfour$ & $\bC(e_3 + i e_1)$ & \begin{tabular}{l}
$e_1 = \Ml
0 & 0 & 0\\
1 & 0 & 0 \\
0 & 0 & 0 \\
\Mr$,\\\addlinespace[2mm]
$e_2 = \Ml
0 & 0 & 0\\
0 & 0 & 0 \\
1 & 0 & 0 \\
\Mr$, \\\addlinespace[2mm]
$e_3 = \Ml
0 & 0 & 0 \\
0 & a & -1 \\
0 & -1 & a \\
\Mr$ 
\end{tabular}&
\begin{tabular}{l}
$[e_1, e_2] = 0$,\\\addlinespace[2mm]
$[e_2,e_3] = e_1 - a e_2$,\\\addlinespace[2mm]
$[e_3,e_1] = a e_1 - e_2$
\end{tabular}\\
\midrule
$\Gfive$ & $\bC(e_3 + i e_2)$ & \begin{tabular}{l}
$e_1 = \Ml
0 & 0 & 0\\
1 & 0 & 0 \\
0 & 0 & 0 \\
\Mr$,\\\addlinespace[2mm]
$e_2 = \Ml
0 & 0 & 0\\
0 & 0 & 0 \\
1 & 0 & 0 \\
\Mr$, \\\addlinespace[2mm]
$e_3 = \Ml
0 & 0 & 0 \\
0 & a & -1 \\
0 & 1 & a \\
\Mr$ 
\end{tabular} &
\begin{tabular}{l}
$[e_1, e_2] = 0$,\\\addlinespace[2mm]
$[e_2,e_3] = e_1 - a e_2$,\\\addlinespace[2mm]
$[e_3,e_1] = a e_1 + e_2$
\end{tabular}\\
\bottomrule
\end{tabular}
\end{minipage}
\end{center}
\end{table}

\subsection{Embeddings}
We next consider the realization of the above CR structures as real hypersurfaces.
\begin{proposition}
All the CR structures in Table \ref{table:2} are embeddable in $\bC^2$. 
\end{proposition}
\begin{proof}
Let $f:\AffR\rightarrow\bC^2$ be defied by $\left(\Ml
1 & 0 \\
x & e^y \\
\Mr, z\right)\mapsto (e^yi+x,z+iy)$. As $df(e_2+ie_1+ie_3)\in T^{1,0}\bC^2$ at the unit element and $f$ is $\AffR$-invariant, we see that $f$ is a CR map. This map embed $\AffR$ into the hypersurface $\frac{\alpha-\bar{\alpha}}{2i}=e^{\frac{\beta-\bar{\beta}}{2i}}$.

The arguments are similar in each case, we only give the embedding map $f$ and its image.

For $\big(\AffS,\bC(e_2+ie_1+ai e_3)\big)$, we may take 
\[
f\left(\Ml
1 & 0 \\
x & e^y \\
\Mr, e^{zi}\right)= (e^yi+x,e^{ay-zi}).
\]
The image is the hypersurface $\frac{\alpha-\bar{\alpha}}{2i}={(\beta\bar{\beta})}^{\frac{1}{2a}}$.

For $\Gtwo$, we have
$ f(g)= (\alpha,\beta)$ with $(1,\alpha,\beta)^\top=g (1,i,i)^\top$. 
The image is the hypersurface $\frac{\alpha-\bar{\alpha}}{2i}=e^{1-\frac{\beta-\bar{\beta}}{\alpha-\bar{\alpha}}}$.

For $\Gfour$, we take
\[
f(g)= (\alpha,\beta) \quad \mathrm{with}\, (1,\alpha,\beta)^\top=g \left(1,-\frac{a}{1-a^2}i,-\frac{1}{1-a^2}i\right)^\top.
\]
Let $h:\bC^2\rightarrow\bC^2$ be defied by $(\alpha,\beta)\mapsto((a-1)(\alpha+\beta),(a+1)(\alpha-\beta))$. Then the image of $h\circ f$ is  the hypersurface $\left(\frac{\alpha-\bar{\alpha}}{2i}\right)^{\left(\frac{1+a}{a-1}\right)}=\frac{\beta-\bar{\beta}}{2i}$ with $\frac{\alpha-\bar{\alpha}}{2i}>0$.

For  $\Gfive$, we take
\begin{equation*}
f(g)=(\alpha,\beta)\quad \mathrm{with}\, (1,\alpha,\beta)^\top=g \left(1,\frac{1}{1+a^2}i,\frac{a}{1+a^2}i\right)^\top.
\end{equation*}
The image is the hypersurface $\frac{\alpha-\bar{\alpha}}{2i}=\frac{\cos c-a\sin c}{1+a^2}e^{ac}$ with 
\[c=\frac{1}{2a}\log\left((1+a^2)\left({\left(\frac{\alpha-\bar{\alpha}}{2i}\right)}^2+{\left(\frac{\beta-\bar{\beta}}{2i}\right)}^2\right)\right).\]
\end{proof}
\begin{remark}
These hypersurfaces appear in the table of Cartan's paper~\cite[p.70]{car}.
\end{remark}
\begin{remark}{(cf.~\cite{bor})}
The CR structures $\left(\SL,a=1,-3+2\sqrt{2}\right)$ are embedded in $\bC^2$. The CR structures $\left(\SL,-1<a<0, a\neq-3+2\sqrt{2}\right)$ are immersed as double covers of the adjoint orbits of $\SL$ in $\mathrm{P}\left(\mathfrak{sl}(2,\bC)\right)$. The CR structures $\left(\SL,0<a<1\right)$ are immersed as $4:1$ covers of the adjoint orbits of $\SL$ in $\mathrm{P}\left(\mathfrak{sl}(2,\bC)\right)$. The Rossi spheres are immersed as $2:1$ covers of the adjoint orbits of $\mathfrak{sl}(2,\bC)$. 
For $(\He, \bC(e_1-ie_2))$, the CR map: $\He\rightarrow\bC^2, \Ml
1 & x & z \\
0 & 1 & y \\
0 & 0 & 1
\Mr\mapsto(\alpha,\beta)=\left(x-iy,4z-2xy+(x^2+y^2)i\right)$ embeds $\He$ in $\mathrm{Im}\,\beta=\alpha\bar{\alpha}$.
\end{remark}

\subsection{Fefferman metrics}
We next compute  Fefferman metrics. Let $\{\omega^i\}$ be the dual coframe of $\{e_i\}$. For each CR structure, take the contact form $\theta$ as in Table \ref{table:3}. 
Then admissible coframes, the corresponding Tanaka--Webster connections/torsions/curvatures are given in Table \ref{table:4}. 
Using these, we obtain Table \ref{table:5} of Fefferman metrics $g_\theta$.

\begin{table}[ht]
\begin{center}
\begin{minipage}{\textwidth}
\caption{Contact forms and corresponding Reeb vector fields}
\label{table:3}
\begin{tabular}{llll}
\toprule
& CR structure & contact form $\theta$ & Reeb vector field $T$\\
\midrule
\addlinespace[2mm]
(A) & 
\begin{tabular}{l}
$(\AffS$,\\
$\bC(e_2+i e_1 + a i e_3))$,\\\addlinespace[2mm]
$a>0$\end{tabular} 
&
$a\omega^1-\omega^3$& 
$-e_3$
\\
\midrule
(B) & $\left(\Gtwo,\bC(e_3 + i e_2)\right)$ & 
$-\omega^1$&
$-(e_1+e_2)$
\\
\midrule
(C) & $\left(\Gfour,\bC(e_3 + i e_1)\right),a\ge0,a\neq1$ & 
$-\omega^2$&
$-(e_2+ae_1)$\\
\midrule
(D) & $\left(\Gfive,\bC(e_3 + i e_2)\right),a\ge0$ & 
$-\omega^1$&
$-(e_1+ae_2)$\\
\midrule
(E) & $\left(\He,\bC(e_1 - i e_2)\right)$ & 
$-\omega^3$& 
$-e_3$\\
\midrule
(F)&
\begin{tabular}{l}
$(\SL$,\\
$\bC\left(ae_1 + e_2 -i\frac{1+a}{2}e_3\right))$, \\\addlinespace[2mm]
$a\in(-1,0)\cup(0,1]$
\end{tabular}
&
$\mathrm{sign}(a)(\omega^1-a\omega^2)$&
$\mathrm{sign}(a)\left(\frac{1}{2}e_1-\frac{1}{2a}e_2\right)$
\\
\midrule
(G) & $\left(\SU,\bC\left(e_1 - \frac{i}{a} e_2\right)\right),a\ge1$ &
$-\omega^3$&
$-e_3$
\\
\bottomrule
\end{tabular}
\end{minipage}
\end{center}
\end{table}

\begin{table}[ht]
\begin{center}
\begin{minipage}{340pt}
\caption{Tanaka--Webster connections}\label{table:4}
\begin{tabular}{llll}
\toprule
& $\theta^1$ & $2h_{1\bar 1}$ & connection/torsion/curvature\\
\midrule
(A) & $\omega^1 + i\omega^2$ & 
$a$&
\begin{tabular}{l}
$\omega_1^{\,\,1}=\frac{i}{2}(\theta^1+\theta^{\bar{1}})$,\\\addlinespace[1mm]
$A_{\bar1}^{\,\,1}=0,S=-\frac{1}{a}$ 
\end{tabular}
\\
\midrule
(B) & 
$\omega^1-\omega^2-i\omega^3$&
$1$&
\begin{tabular}{l}
$\omega_1^{\,\,1}=-\frac{i}{2}\theta -i(\theta^1+\theta^{\bar{1}})$,\\\addlinespace[1mm]
$A_{\bar1}^{\,\,1}=-\frac{i}{2},S=-\frac{7}{2}$
\end{tabular}\\
\midrule
(C) & 
$\omega^1-a\omega^2+i\omega^3$& 
$1$&
\begin{tabular}{l}
$\omega_1^{\,\,1}=\frac{a^2-1}{2i}\theta+a i(\theta^1+\theta^{\bar{1}})$,\\\addlinespace[1mm]
$A_{\bar1}^{\,\,1}=\frac{a^2-1}{2i},S=-\frac{7a^2+1}{2}$
\end{tabular}\\
\midrule
(D) & 
$a\omega^1-\omega^2-i\omega^3$&
$1$
&
\begin{tabular}{l}
$\omega_1^{\,\,1}=\frac{a^2+1}{2i}\theta-a i(\theta^1+\theta^{\bar{1}})$,\\\addlinespace[1mm]
$A_{\bar1}^{\,\,1}=\frac{a^2+1}{2i},S=\frac{1-7a^2}{2}$ 
\end{tabular}\\
\midrule
(E) & 
$\omega^1+i\omega^2$& 
$1$ &
\begin{tabular}{l}
$\omega_1^{\,\,1}=0$,\\\addlinespace[1mm]
$A_{\bar1}^{\,\,1}=0,S=0$ 
\end{tabular}\\
\midrule
(F) & 
$\sqrt{\frac{4\lvert a \rvert}{a+1}}
\left(\omega^3-i\frac{1+a}{2}\omega^2-\frac{1+a}{4\lvert a \rvert}i\theta\right)$
& 
$2$&
\begin{tabular}{l}
$\omega_1^{\,\,1}=-Si\theta$,\\\addlinespace[1mm]
$S=-\frac{1+6a+a^2}{4\lvert a \rvert(1+a)}$,\\\addlinespace[1mm]
$A_{\bar1}^{\,\,1}=-i\frac{(1-a)^2}{4\lvert a \rvert(1+a)}$ 
\end{tabular}\\
\midrule
(G) & 
$\omega^1+ai\omega^2$
& 
$1/a$
&
\begin{tabular}{l}
$\omega_1^{\,\,1}=-Si\theta,S=\frac{1+a^2}{2a}$,\\\addlinespace[1mm]
$A_{\bar1}^{\,\,1}=\frac{1-a^2}{2ai}$ 
\end{tabular}\\
\bottomrule
\end{tabular}
\end{minipage}
\end{center}
\end{table}

\begin{table}[ht]
\begin{center}
\begin{minipage}{300pt}
\caption{Fefferman metrics}\label{table:5}
\begin{tabular}{ll}
\toprule
 & Fefferman metric\\
\midrule
(A)
& $g_\theta=a(\omega^1\otimes\omega^1+\omega^2\otimes\omega^2)+\frac{4}{3}\left(ds-\omega^1+\frac{1}{4a}\theta\right)\odot\theta$\\
\midrule
(B)  &
$g_\theta=\frac{1}{6}\omega^1\otimes\omega^1+\frac{2}{3}\omega^1\odot\omega^2+\omega^2\otimes\omega^2+\omega^3\otimes\omega^3-\frac{4}{3}\omega^1\odot ds$\\
\midrule
(C) & $g_\theta= \omega^1\otimes\omega^1 + (\frac{2a}{3}\omega^1-\frac{4}{3}ds)\odot\omega^2+\frac{a^2-3}{6}\omega^2\otimes\omega^2+\omega^3\otimes\omega^3$\\
\midrule
(D) 
& $g_\theta=\frac{a^2+3}{6}\omega^1\otimes\omega^1+\omega^2\otimes\omega^2+(\frac{2a}{3}\omega^2-\frac{4}{3}ds)\odot\omega^1+\omega^3\otimes\omega^3$\\
\midrule
(E) & $g_\theta=\omega^1\otimes\omega^1+\omega^2\otimes\omega^2-\frac{4}{3}\omega^3\odot ds$\\
\midrule
(F) & $g_\theta=2\theta^1\odot\theta^{\bar{1}}+\left(-\frac{1+6a+a^2}{4\lvert a \rvert(1+a)}\theta+\frac{4}{3}ds\right)\odot\theta$\\
\midrule
(G) & $g_\theta=\frac{1}{a}\omega^1\otimes\omega^1+a\omega^2\otimes\omega^2+4\omega^3\odot\left(\frac{1}{8}\left(a+\frac{1}{a}\right)\omega^3 - \frac{1}{3} ds \right)$\\
\bottomrule
\end{tabular}
\end{minipage}
\end{center}
\end{table}

\section{Chains of left-invariant CR structures on the 3-dimensional connected Lie groups except for $\SU$}\label{sec4}
In this section, we prove Theorem \ref{mainthm} except for the case of $\SU$. For any left-invariant CR structure on a Lie group $G$, the Fefferman space is $G\times S^1$ and $g_{\theta}$ is a left-invariant metric on $G\times S^1$. Thus we can apply the formulation in Section \ref{sec2.2} to the Fefferman metrics. When two Lie groups has the same Lie algebra, we can use the following

\begin{lemma}\label{covering}
Let $G$ be Lie groups with left-invariant nondegenerate CR structure. 
If $\tilde{G}$ is a covering group of $G$ with lifted CR structure, then the projections of the chains on $\tilde{G}$ are the chains of $G$.
\end{lemma}

Except for the case (F),  we take $\{\omega^1,\omega^2, \theta, ds\}$ as a coframe of $G\times S^1$, 
which can be also seen as a basis of the dual space of the Lie algebra of $G\times S^1$.
Then  the Hamiltonian and the Lie--Poisson equation are respectively given by \eqref{Hamiltonian} and \eqref{L-Peq}, in terms of the trivialization $M_{j}$.
Note that  the null geodesics lie on  the zero locus of the Hamiltonian $H=0$. Also, each solution of the Lie-Poisson equations lie on an orbit of the coadjoint action because of the momentum map $J$. 

We use the following criterions to obtain non-closed chains.
If we can find a non-closed solution of the Lie--Poisson equation, then the corresponding chain is not closed.  If we can find a stationary solution $(M_{j})$, then the OED for the null geodesic is
$
\dot p=M^{j}e_{j}
$ and 
the equation for the chain (which is also denoted by $p(t)$) is then obtained by omitting $e_{4}=\frac{\partial}{\partial s}$ component.
We will show that  these chains are not closed.

\subsection{The Lie--Poisson equation for (A)}
The Lie--Poisson equation and the Hamilton function are 
\begin{equation}
\label{L-PeqA}
\left\{
\begin{array}{llll}
\dot{M_1} = M_2 M_3 + \frac{1}{a} M_1M_2\\
\dot{M_2} = -\frac{1}{a} M_1^2 - M_1M_3 - \frac{1}{a} M_1M_4 - M_3M_4\\
\dot{M_3} = 0 \\
\dot{M_4} = 0 \\
\end{array}
\right.
\end{equation}
\[
2H =\frac{1}{a}\left({(M_1+M_4)}^2+M_2^2\right)+3M_3M_4-\frac{3}{4}\frac{M_4^2}{a}.
\]
Hence $M_{3}$ and $M_{4}$ are constant along the geodesics.
We set $M_3=-\frac{1}{2a}, M_4=1$ and consider the solutions lie on
$H=0$,  i.e., $(M_1+1)^2+M_2^2=\frac{9}{4}$. The point $(M_1,M_2)=\left(\frac{1}{2},0\right)$ is the only stationary point, where the right-hand sides of (\ref{L-PeqA}) are $0$, so those solutions are homoclinic orbits (see e.g.~\cite{guc} about homoclinic orbits), especially these are not closed. 
The chain corresponding to these orbits are no closed as $M_{1}$ is not periodic. Therefore, we have done for  $\AffS$. By Lemma \ref{covering}, the case for $\AffR$ is also done.

\subsection{The Lie--Poisson equation for (B)}
The equations are
\[
\left\{
\begin{array}{llll}
\dot{M_1} = M_1 M_3\\
\dot{M_2} = - M_1M_3 + M_2M_3\\
\dot{M_3} = 2M_1M_4 + M_1M_2 - M_2^2 - \frac{1}{2} M_2M_4 \\
\dot{M_4} = 0 \\
\end{array}
\right.
\]
\[
2H = -3M_1M_4 +M_2^2 + M_3^2 + M_2M_4 -\frac{1}{8}M_4^2.
\]
Let $M_1=0$ and $M_4=1$. Consider the solutions lie on $H=0$, that is, on the curve $\left( M_2+\frac{1}{2} \right)^2 +M_3^2 = \frac{3}{8}$ in $\bR^{2}$.
The points $(M_2,M_3)=\left(0,\pm\frac{1}{2\sqrt{2}}\right)$ are the stationary points, 
so we have  heteroclinic orbits on this circle. Thus we have non closes chains.

\subsection{The Lie--Poisson equation for (C)}
The equations are
\[
\left\{
\begin{array}{llll}
\dot{M_1} = a M_1 M_3 - M_2M_3\\
\dot{M_2} = a M_2 M_3 - M_1M_3\\
\dot{M_3} = -\frac{a^2+3}{2}M_1M_4 - a M_1^2 + 2a M_2M_4 + M_1M_2\\
\dot{M_4} = 0 \\
\end{array}
\right.
\]
\[
2H = M_1^2+a M_1M_4 - 3M_2M_4 + M_3^2 +\frac{9-a^2}{8}M_4^2.
\]
When $M_3=0, M_4=1$, a solution of the system of the equations $-\frac{a^2+3}{2}M_1 - a M_1^2 + 2a M_2 + M_1M_2=0, M_1^2+a M_1 - 3M_2 +\frac{9-a^2}{8}=0$ is a stationary point. Then it follows from (\ref{L-Peq}) that the equation of chains is $\dot{p}=\left(M_1+\frac{a}{2}\right)e_1-\frac{3}{2}e_2$. 
The solution 
\begin{equation*}
p(t)=\begin{pmatrix}
1 & 0 & 0\\
\left(M_1+\frac{a}{2}\right)t & 1 & 0 \\
-\frac{3}{2}t & 0 & 1
\end{pmatrix}
\end{equation*}
is not closed. 

\subsection{The Lie--Poisson equation for (D)}
For $\Gfive$, the equations are
\[
\left\{
\begin{array}{llll}
\dot{M_1} = a M_1 M_3 + M_2M_3\\
\dot{M_2} = a M_2 M_3 - M_1M_3\\
\dot{M_3} = \frac{3-a^2}{2}M_2M_4 - a M_2^2 + 2a M_1M_4 + M_1M_2\\
\dot{M_4} = 0 \\
\end{array}
\right.
\]
\[
H = \frac{1}{2}\left(M_2^2+a M_2M_4 - 3M_1M_4 + M_3^2 -\frac{9+a^2}{8}M_4^2\right).
\]
When $M_3=0, M_4=1$, the solutions of the system of equations $\frac{3-a^2}{2}M_2 - a M_2^2 + 2a M_1 + M_1M_2=0, M_2^2+a M_2 - 3M_1 -\frac{a^2+9}{8}=0$ give  stationary points. 
For such points, the equation of chains is $\dot{p}=-\frac{3}{2}e_1 + \left(M_2+\frac{a}{2}\right)e_2$,
which has a non-closed solution
\[
p(t)=\begin{pmatrix}
1 & 0 & 0\\
-\frac{3}{2}t & 1 & 0\\
\left(M_2+\frac{a}{2}\right)t & 0 & 1 \\
\end{pmatrix}.
\]
Since $\SEU$ and $\SEn$ cover  $\Gfive$, we also see that these cases also have non-closed chains.

\subsection{The Lie--Poisson equation for (E)}\label{chainofH}
While it is well-known, we compute the chains on $\He$ for the completeness.
The geodesic equation and the Hamilton function are
\[
\left\{
\begin{array}{lllll}
\dot{M_1} = -M_2 M_3\\
\dot{M_2} = M_1M_3\\
\dot{M_3} = 0 \\
\dot{M_4} = 0 \end{array}
\right.
\]
\[
2H =M_1^2 + M_2^2 - 3M_3M_4.
\]
For a solution $\left(p(t),\left\{M_i(t)\right\},M_4\right)$ with $M_4\neq0$, 
\[
\left(p(t/M_4),\left\{M_i(t/M_4)/M_4\right\},1\right)
\] is also a solution. 
Thus, noting that $\theta(\dot{p})=-M_{4}\ne0$, by a linear change of  the time parameter $t$, we may assume $M_{4}=1$.
Thus,
\[(M_1,M_2)=\Big(\sqrt{3M_3}\cos(M_3t), \sqrt{3M_3}\sin(M_3t)\Big)\] is a solution of the Lie--Poisson equation.

The equation for the chain is then given by 
\[
\dot{p}=M_1e_1+M_2e_2-\frac{3}{2}M_4e_3-\frac{3}{2}M_3.
\]
 If $M_3\neq0$, the chain $p(t)$ with $p(0)=\Ml
1 & 0 & 0\\
0 & 1 & 0\\
0 & 0 & 1\\
\Mr$ is 
\begin{equation*}
p(t)=\begin{pmatrix}
1 & \frac{\sqrt{3M_3}}{M_3}\sin(M_3t) & -\frac{3}{4M_3}\sin(2M_3t)\\
0 & 1 & \frac{\sqrt{3M_3}}{M_3}(1-\cos(M_3t))\\
0 & 0 & 1
\end{pmatrix}.
\end{equation*}
If $M_3=0$, the chain $p(t)$ can be written in the form 
\begin{align}
\label{ChainH}
p(t)=p(0)\begin{pmatrix}
1 & 0 & ct\\
0 & 1 & 0\\
0 & 0 & 1\\
\end{pmatrix}
\end{align}
with some constant $c$. 
If we take $\bZ(\He(\bZ))$-quotient, then $p(t)$ projects to a closed courve. Thus we see that 
 all the chains on
 $\HQ$ are closed.
 
 \subsection{The Lie--Poisson equation for (F)}
Let $\{\tilde{e_1},\tilde{e_2},\tilde{e_3}\}$ be the frame 
\[\left\{ae_1+e_2,\frac{1+a}{2}e_3, \left(\frac{1}{2}e_1-\frac{1}{2a}e_2\right)\mathrm{sign}(a)\right\}.\]
If we take $\left\{\tilde{e_1},\tilde{e_2},\tilde{e_3}, \frac{\partial}{\partial s}\right\}$ as a 
basis of the Lie algebra of $\SL\times S^1$, then the Lie--Poisson equation and the Hamilton function are
\[
\left\{
\begin{array}{lllll}
\dot{M_1} = M_2 M_3 + \frac{3}{1+a} \mathrm{sign}(a) M_2 M_4 \\
\dot{M_2} = - M_1 M_3 - \frac{3(1+a)}{4\lvert a \rvert} M_1 M_4\\
\dot{M_3} = (\frac{1}{4a^2} - \frac{1}{a(1+a)^2}) M_1 M_2 \\
\dot{M_4} = 0.
\end{array}
\right.
\]
\[
2H = \frac{{M_1}^2+{M_2}^2}{2\lvert a \rvert(1+a)}+3M_3M_4+\frac{9}{16}\frac{a^2+6a+1}{\lvert a \rvert(1+a)}M_4^2.
\]
Let $M_4=1$. Then the solution curves are on the paraboloid $H = 0$ with quadric surfaces 
\[
\frac{{M_1}^2}{4a^2} +\frac{{M_2}^2}{a{(1+a)}^2}-{M_3}^2=K
\]
 for some constant $K$. Those quadric surfaces are orbits of the coadjoint action of $\SL$. For $-1<a\le5-2\sqrt{6}, a \neq 0$, the  solution corresponding to $K=-\frac{9}{256}{\left(\frac{a^2+6a+1}{a(1+a)}\right)}^2$ is homoclinic. 
 For $5-2\sqrt{6}<a<1$, the solution corresponding to $K=-\frac{9}{8}\frac{(1-a)^2}{a(1+a)^2}$ is heteroclinic.
When $a=1$, the equation for the chain is
\[
\dot{p} = \frac{1}{4}M_1\tilde{e_1} + \frac{1}{4}M_2\tilde{e_2} + \frac{3}{2} M_4\tilde{e_3},
\]
which has non-closed solution.
The non-closed chains on $\SL$ project to non-closed chains on $\PSL$, and these are lifted to 
non-closed chains on the coverings $\An$ and $\AU$.

\section{Chains on $\SU$}\label{sec5}
\subsection{The Lie--Poisson equation}
For the case $\SU\times S^1$, the equations are
\begin{equation}
\label{L-PeqSU}
\left\{
\begin{array}{lllll}
\dot{M_1} = - \frac{1}{a} M_2 M_3 - \frac{3}{2} M_2 M_4 \\
\dot{M_2} = a M_1 M_3 + \frac{3}{2} M_1 M_4\\
\dot{M_3} = (\frac{1}{a} - a) M_1 M_2 \\
\dot{M_4} = 0 
\end{array}
\right.
\end{equation}
\[
2H = a{M_1}^2+\frac{1}{a}{M_2}^2-3M_3M_4-c(a)M_4^2,
\]
where
$c(a)=\frac{9}{8}\left(a+\frac{1}{a}\right)$.
As in Section \ref{chainofH}, we may assume $M_4=1$. The solution curves are the intersections of the paraboloid $H = 0$ with the spheres $K = {M_1}^2+{M_2}^2+{M_3}^2$, where $K$ is some constant. In~\cite{cas}, the projections of the curves onto $M_1M_2$-plane are classified as follows.

For $a>\sqrt{3}$, the origin becomes a saddle point. The curve corresponding to the case $K=\frac{9}{64}{\left(\frac{a+1}{a}\right)}^2$ is homoclinic. For $\frac{9}{8}\left(1-\frac{1}{a^2}\right)\le K<\frac{9}{64}{\left(\frac{a+1}{a}\right)}^2$, the curve is consisted of the two disjoint closed curves, and the curve becomes one and surround the origin for $K>\frac{9}{64}{\left(\frac{a+1}{a}\right)}^2$. 
By these observations, we see that there are non-closed chains for $a>\sqrt{3}$.

For $1\le a\le\sqrt{3}$, all the curves are closed and surround the origin. 
 For $1\le a\le\sqrt{3}$, the sphere intersects with the paraboloid when $K\ge{c(a)}^2/9$. The projections of the solution curves $M(t)$ onto $M_1M_2$-plane are symmetric to the $M_1$-axis and $M_2$-axis.
 To describe the curve let us set
 \[
\begin{array}{rl}
\alpha &= \frac{1}{2}\left(3-3a^2+\sqrt{9a^4-4c(a)a^3+4Ka^2}\right),\\
\beta &=\frac{1}{2}\left(3-3a^2-\sqrt{9a^4-4c(a)a^3+4Ka^2}\right),\\
\gamma &= \sqrt{Ka^2-c(a)a+{9}/{4}}.
\end{array}
\]
It is worth noting that $\beta\le\alpha\le\gamma$.
Then, in the first quadrant, $M_1$ is represented by $M_2$ as
\begin{equation*}
M_1 =  \frac{1}{a}\sqrt{-{M_2}^2+c(a)a-{9}/{2}+3s(M_{2})}\quad
\text{for } 0 \le M_2 \le M_2^{\mathrm{max}}, 
\end{equation*}
where 
\begin{align*}
M_2^{\mathrm{max}}&=\sqrt{c(a)a-{9}/{2}+3\alpha},
\\
s(M_{2})&=\sqrt{(1-a^2){M_2}^2+\gamma^2}.
\end{align*}
Then the period of $M(t)$ is given by
\begin{equation*}
T(K,a) =4 \int_0^{M_2^{\mathrm{max}}} \frac{dM_2}{\left(aM_3+{3}/{2}\right)M_1} .
\end{equation*}
Using $H=0$, we may write
$
aM_3+\frac{3}{2}=s(M_{2})$,
and the denominator in the integral can be written as
\[
\frac{s(M_{2})}{a}\sqrt{-{M_2}^2+c(a)a-{9}/{2}+3s(M_{2})}.
\]
Hence the change of the variable gives
\begin{equation}\label{period}
T(K,a) = \int_\alpha^\gamma\frac{4a \,ds}{\sqrt{(s-\alpha)(s-\beta)}\sqrt{\gamma^2-s^2}}.
\end{equation}

\begin{lemma}
For $1\le a\le\sqrt{3}$, one has
\begin{equation*}
\lim_{K\rightarrow {c(a)}^2/9+0} T(K,a) = \frac{16a\pi}{3\sqrt{3-a^2}\sqrt{3a^2-1}}.
\end{equation*}
When $a=\sqrt{3}$, the right-hand side is read as $\infty$.
\end{lemma}

\begin{proof}
As $K\rightarrow{c(a)}^2/9+0$, we have
\[
\alpha\to \frac{3}{2}-\frac{c(a)}{3},\quad \gamma\to \alpha\to \frac{3}{2}-\frac{c(a)}{3},
\quad
 \beta \to -3a^2+\frac{c(a)a}{3}+\frac{3}{2}.
\]
On the other hand, \eqref{period} and 
\[\int_\alpha^\gamma \frac{ds}{\sqrt{s-\alpha}\sqrt{\gamma-s}}=\pi
\]
give the estimate
\begin{equation*}
\frac{4a\pi}{\sqrt{\gamma-\beta}\sqrt{2\gamma}}\le T(K,a)\le\frac{4a\pi}{\sqrt{\alpha-\beta}\sqrt{\alpha+\gamma}}.
\end{equation*}
Both sides converge to the value claimed in the lemma.
\end{proof}

Similarly, we obtain the following
\begin{lemma}\label{TKinfty}
For $1\le a\le\sqrt{3}$, one has $\lim_{K\rightarrow\infty} T(K,a) = 0.$
\end{lemma}

\subsubsection{Computation of the Berry phase}

For each closed solution $M(t)$ of the Lie--Poisson equation with period $T=T(K,a)$, we  study the equation for the chain on $\SU$:
\[
\dot{p} = aM_1e_1 + \frac{1}{a}M_2e_2 - \frac{3}{2} e_3.
\]
Recall that the orbit $C$ of the curve $M(t)$ in $\mathfrak{su}(2)^*\cong\bR^{3}$ is the intersection of $H=0$ with the sphere of radius $\sqrt{K}$;
let us denote by $D(K,a)$ is the part of the sphere below the curve $M$.

To simplify the notation, let us write $G=\SU$ and $J\colon T^{*}G\to\mathfrak{g}^*$ be   the momentum map and $\pi\colon T^{*}G\to\mathfrak{g}^*$ be given by $M$.
 (Our notation is different from the one in~\cite{cas}, where $G=\SU\times S^1$). 
It is then shown that
$J^{-1}(\mu)\cap \pi^{-1}(C)$ is a 2-torus $T^{2}$, and the solution $\gamma(t)=(p(t),M(t))$ lies there.
Since $\pi\colon T^{2}\to C$ is a $S^{1}$-bundle, we may define the angle $\Delta\theta$ between
$\gamma(0)$ and $\gamma(T)$, which is called the 
the \textit{Berry phase}.  A formula for  $\Delta\theta$ has been given in 
~\cite[(20)-(22)]{cas}:
\begin{align}
\label{Berryphase}
\Delta\theta(K,a)=\frac{1}{\sqrt{K}}\int_0^{T(K,a)} f(t,K,a) dt +\text{[Area of }D(K,a)],\end{align}
where 
\[
f(t,K,a) = \frac{3}{2}M_3(t) + c(a) = \frac{1}{2}\left(a{M_1(t)}^2+\frac{1}{a}{M_2(t)}^2 + c(a)\right).
\]

\begin{proposition}
For $1\le a\le\sqrt{3}$, one has
\begin{equation*}
\lim_{K\rightarrow \infty} \Delta\theta(K,a) = 4\pi.
\end{equation*}
\end{proposition}
\begin{proof}
The first term of \eqref{Berryphase} satisfies
\begin{equation*}
0 \le \frac{1}{\sqrt{K}}\int_0^{T} \frac{3}{2}M_3(t) + c(a) dt \le \frac{T}{\sqrt{K}}\left(\frac{3}{2}\sqrt{K} + c(a)\right).
\end{equation*}
By using Lemma \ref{TKinfty}, we can see that the right-hand side converges to 0 when $K\rightarrow\infty$. 

It remains to show that $D(K,a)$ converges to $4\pi$ as $K\to\infty$.
Since $M_3=\frac{1}{a}\left(-3/2+\sqrt{(1-a^2){M_2}^2+\gamma^2}\right)$, we see that $M_3$ takes the minimum value $M_{3}^{\mathrm{min}}$ at the point $(M_1,M_2)=(0,M_{2}^{\mathrm{max}})$ 
on  $C$. Therefore, 
\[\lim_{K\rightarrow\infty}M_{3}^{\mathrm{min}}/\sqrt{K} = \lim_{K\rightarrow\infty}\gamma/(a\sqrt{K}) = 1.
\]
\end{proof}

\begin{proposition}
For $1\le a\le\sqrt{3}$, one has
\begin{equation*}
\lim_{K\rightarrow {c(a)}^2/9+0} \Delta\theta(K,a) = \frac{8a\pi}{\sqrt{3-a^2}\sqrt{3a^2-1}}.
\end{equation*}
For $a=\sqrt{3}$, the right-hand side is read as $\infty$.
\end{proposition}
\begin{proof}
When $K\to{c(a)}^2/9+0=:\lambda$, $M_1$ and $M_2$ converge to $0$ uniformly in $t$. 
Thus the first term of \eqref{Berryphase} converges to
\[
\frac{1}{2}c(a)\lim_{K\rightarrow\lambda} \frac{T(K,a)}{\sqrt{K}}=
\frac{3}{2}\lim_{K\rightarrow \lambda} T(K,a)= \frac{8a\pi}{\sqrt{3-a^2}\sqrt{3a^2-1}}.
\]
 The second term of \eqref{Berryphase} converges to $0$.
\end{proof}
For each fixed $1<a\le\sqrt{3}$, the above two limit values of $\Delta\theta(K,a)$ are different. Thus $\Delta\theta(K,a)$ is not constant. In particular,  not all the chains are closed.

In~\cite[Proposition 7.1]{cas}, it was shown that not all the chains are closed for any $a>\sqrt{3}$. Thus for any left-invariant aspherical CR structure on $\SU$, both closed chains and quasiperiodic chains exist. Any quasiperiodic chain is dense on a two-tori in $\SU$ and $\SO=\SU/\{\pm I\}$, so the projections onto $\SO$ of the quasiperiodic chains are not closed.
Thus we have completed the proof of Theorem \ref{mainthm}.

\section*{Acknowledgments}
This paper is based on my Master's thesis.
I would like to thank Professor Kengo Hirachi for his guidance and constructive suggestions.

\end{document}